\documentclass[a4paper,11pt]{article}

\usepackage[a4paper,margin=3.5cm]{geometry}
\usepackage[T1]{fontenc}
\usepackage[cm]{aeguill}
\usepackage{amsfonts,amsthm,amssymb,amsmath}
\usepackage[utf8]{inputenc}
\usepackage{amssymb}
\usepackage{dsfont}
\usepackage{mathrsfs}

\newtheorem{Th}{Theorem}

\newtheorem{Def}{Definition}
\newtheorem{Prop}{Proposition}
\newtheorem{Cor}{Corollary}

\title{A coupling construction for spin systems with infinite range interactions}
\author{F. Ezanno\footnote{\noindent CMI, 
Universit\'e de Provence, 39, rue F. Joliot Curie, 13453 Marseille Cedex 13, FRANCE, \textit{fezanno@cmi.univ-mrs.fr} \newline \textit{AMS} 2010 \textit{subject classification}. Primary, 60K35. \newline \textit{Keywords and phrases.} Interacting particle systems, graphical construction, coupling. }}
\date{}

\newcommand{\RR}{\mathbb{R}}
\newcommand{\ZZ}{\mathbb{Z}}
\newcommand{\EE}{\mathbb{E}}
\newcommand{\PP}{\mathbb{P}}

\newcommand{\Pg}{\mathbf{P}}
\newcommand{\Eg}{\mathbf{E}}
\newcommand{\F}{\mathcal{F}}

\newcommand{\ind}{\mathds{1}}
\newcommand{\DX}{\mathcal{D}(X)}
\newcommand{\CX}{\mathcal{C}(X)}

\newcommand{\ud}{\mathrm{d}}

\newcommand{\gui}[1]{``#1''}

\def\restriction#1#2{\mathchoice
{\setbox1\hbox{${\displaystyle #1}_{\scriptstyle #2}$}
\restrictionaux{#1}{#2}}
{\setbox1\hbox{${\textstyle #1}_{\scriptstyle #2}$}
\restrictionaux{#1}{#2}}
{\setbox1\hbox{${\scriptstyle #1}_{\scriptscriptstyle #2}$}
\restrictionaux{#1}{#2}}
{\setbox1\hbox{${\scriptscriptstyle #1}_{\scriptscriptstyle #2}$}
\restrictionaux{#1}{#2}}}
\def\restrictionaux#1#2{{#1\,\smash{\vrule height .8\ht1 depth .85\dp1}}_{\,#2}} 
\def\tn{|\!|\!|}

\begin{document}
\maketitle
\begin{abstract}
Given a countable set of sites and a collection of flip rates at each site, we give a sufficient condition on the long-range dependancies of the flip rates ensuring the well-definedness of the corresponding spin system. This hypothesis has already been widely used (\cite{Gray}, \cite{GrayGriffeath}, \cite{HolleyStroock}, \cite{Liggettarticle}) but our construction brings a new insight to understand why it is natural.
\\The process is first constructed as a limit of finite spin systems. Then we identify its generator and give a simple criterion for a measure to be invariant with respect to it.
\end{abstract}

In this paper we propose a new construction of a class of Markov processes, called spin systems, taking values in a configuration space $X=\{0,1\}^V$ where $V$ is a countable set whose elements are called \textit{sites}. For these processes the dynamics are such that when a transition occurs only one coordinate is allowed to change, and for an initial configuration $\eta$ the coordinate $\eta(v)$ flips at some rate denoted by $c_v(\eta)$. By this we mean that when $t\to 0$, we require that $$\PP_\eta\big(\text{state at site }v\text{ at time }t\text{ is } 1-\eta(v)\big)=c_v(\eta)t+o(t).$$ Thus the dynamics are governed by a family $$ c =(c_v(\eta),~ v\in V,~ \eta\in X)$$ of nonnegative real numbers called \emph{flip rates}. The content of this paper could be generalized in two possible directions: we could allow each coordinate to take values in some countable set, taking for instance $X=\ZZ^V$, and we could also consider transitions that may affect the states of a finite (but bounded) number of sites instead of affecting just one site. However these two extensions could be treated with a technique similar to the one presented here, and this would complicate the notations without making the problem more interesting.% and we decide to restrict ourselves to these particular settings.
\\The problem of defining in a satisfactory way this class of processes is not recent. Liggett, in \cite{Liggettarticle} and later in \cite{Liggett} (Chapter 1), gave a quite general solution using an approach based on the Hille-Yosida theorem. In \cite{Gray}  and \cite{HolleyStroock} the problem was considered in terms of the well-posedness of the corresponding martingale problem. For each of these approaches it turns out to be necessary to require good conditions for the family of flip rates $ c$, that loosely speaking must ensure that $c_v(\eta)$ depends on the coordinates $\eta(w)$ in a \gui{sufficiently decreasing} way as $w$ goes to infinity. This corresponds to our condition (\ref{cond2}), which already appeared in various papers, including \cite{Gray} (condition 4.10), \cite{GrayGriffeath} (condition 14), \cite{HolleyStroock} (condition 4.6) and \cite{Liggettarticle} (condition 1.6). 
\\The main purpose of this paper is to give an intuitive explanation of (\ref{cond2})  and simultaneously provide a quite general construction based on elementary probabilistic tools. In this regard, our construction is in the same spirit as Harris' in \cite{Harris} but we do not require a finite range interaction. We show that when it is satisfied a certain countable state Markov process is non-explosive. This implies that there is no influence coming from infinity at a given site in finite time. Heuristically, the main idea here is that in order to know the state $\xi^\eta_t(v)$  of our spin system at site $v$ and time $t$ starting from some configuration $\eta$ at time $0$, one has to know $\eta(w)$ for a certain (random) set $E$ of sites $w$. The interpretation of (\ref{cond2}) given in this paper is that it yields the appropriate control over the long range dependancies of $c_v(\eta)$ for $E$ to remain almost surely finite, so actually the fact of \gui{freezing} all the sites outside a finite box has no influence on the value of $\xi^\eta_t(v)$ provided that this box is big enough. This is the program of Section 1 that leads to the limit process of finite spin systems. Section 2 is dedicated to make the link between this limit process and the Markov pregenerator associated to $c$. It is of practical importance, given a construction of some process, to have a characterisation of its invariant measures. This is also done in Section 2.

%%-------------------------------------------
\section{Coupling and existence of the limit process}  %%-----
%%-------------------------------------------

The space $X=\{0,1\}^V$ being endowed with its product topology, we recall that $X$ is a metrizable compact space. Any element $\eta$ of $X$ is canonically identified with a function $\eta : V\to\{0,1\}$, and with the subset $\{v\in V:\eta(v)=1\}$ of $V$. The sites are denoted by Latin letters such as $v,w$ and the configurations by Greek letters such as $\xi,~\zeta,~\eta$ or $\chi$. Let  $X'=\{\eta\in X : \mathrm{Card}(\eta)<+\infty\}$, where $\mathrm{Card}(\eta)$ denotes the cardinality of $\eta$, and for any $\eta\in X $ and $v\in V$ let $\eta^v$ be the configuration that coincides with $\eta$ everywhere off $v$ but not at $v$:
$$\eta^v(w)=\begin{cases} \eta(w),&\text{ if  $w\neq v$,}\\
1-\eta(v),&\text{ if $w=v$.}\end{cases}
$$
Let us write $1_W$ for the configuration such that $1_W(v)=1$ if $v\in W$ and $0$ otherwise. From now on we fix a collection of flip rates $ c =(c_v(\eta),~ v\in V,~ \eta\in X)$, and we assume that each $c_v$ is a continuous function on $X$. For any pair $(w,v)\in V^2$ of sites, with $w\neq v$, let $$a(w,v)=\sup_{\eta\in X}|c_v(\eta)-c_v(\eta^w)|,$$ and let $a(w,w)=0$. One may think of $a(w,v)$ as a measure of the maximal influence of the state at site $w$ on the flip rate at site $v$. By the continuity of $c_v$ it is clear that whenever two configurations $\eta_1$ and $\eta_2$ coincide on some subset $W$ of $V$ we have 
\begin{equation}|c_v(\eta_1)-c_v(\eta_2)|\leq\sum_{w\in W^c}a(w,v).\label{coincide}\end{equation}
We introduce the notation $\overline a(w,v)=a(v,w)$. Throughout this paper we will assume that the flip rates are bounded:

\begin{equation}\tag{H1}
 C=\sup_{v\in V,\eta\in X}c_v(\eta)<+\infty, 
 \label{cond1}
\end{equation}
and that there exists a family $(\lambda_v,v\in V)$ bounded away from $0$: $$\lambda=\inf_{v\in V}\lambda_v>0,$$ such that
\begin{equation}\tag{H2}
A=\sup_{v\in V}\sum_{w\neq v}\frac{\lambda_w}{\lambda_v}a(w,v)<+\infty.
\label{cond2}
\end{equation} 
As already mentioned, various papers already used the last assumption but we point out that a slightly stronger version of (\ref{cond2}) is stressed in \cite{Liggettarticle} and \cite{HolleyStroock} where only the special case $\lambda_v=1$ is considered. As we said the interest of our construction is to enlighten the natural aspect of (\ref{cond2}), using essentially coupling methods and a duality relation between two processes. To sum up our strategy, we first define a process by letting evolve only the coordinates lying in a finite box $V_n\subset V$, then we let this box grow to $V$, and finally we prove that the obtained sequence of processes converges. Our Theorem \ref{couplage} makes all of this possible with a graphical construction. 
\\Before explaining how to obtain a spin system, let us begin by defining finite spin systems. These are simple Markov jump processes on a finite set of configurations.  
\begin{Def}
Let $\eta\in X$ and $W$ be a finite subset of $V$. We say that a process $\xi=(\xi_t,t\geq 0)$ is a $(\eta,W, c )$-\emph{finite spin system} if it is a Markov process on $X^W_\eta=\{\theta\in X : \restriction{\theta}{W^c}=\restriction{\eta}{W^c} \}$ such that 
\begin{itemize}
\item $\xi_0=\eta$;
\item for any $\theta,\theta'\in X^W_\eta$, its jump rate from $\theta$ to $\theta'$ is $c_v(\theta)$ if $\theta'=\theta^v$ for some $v\in W$, and $0$ otherwise.
\end{itemize}
We denote by $\PP_{\eta,W}$ the law of such a process and $\EE_{\eta,W}$ the expectation under that law. When we will use these notations we will continue to denote by $\xi_t$ the value of the process at time $t$.
\end{Def}
Let us choose an increasing sequence $(V_n,~n\geq 1)$ of finite boxes such that $\cup_{n\geq 1}V_n=V$. We want to describe what happens to a $(\eta,V_n, c )$-\emph{finite spin system} when $n$ goes to $\infty$.
\\We now turn to the definition of another class of Markov processes, called \textit{invasion processes}, that we will use as a tool to control the effect of changing one of the parameters $\eta$ or $W$ on the evolution of our finite spin system.
\\Let $W\subset V$ and $\alpha=(\alpha(w,v),~w\neq v)$ be a family of nonnegative numbers. Later, the role of $\alpha$ will be played by $a$ or $\overline a$. The principle is that for each pair $(x,y)$ of sites, we place independently arrows from $x$ to $y$ at the jump times of a Poisson process, and we decide that the 1's spread using these arrows. More precisely we consider a family of mutually independant Poisson processes $(P_{x,y},~x\neq y)$, $P_{x,y}$ having the intensity $\alpha(x,y)$, and we define a (random) oriented graph $G$ on the set $V\times\RR_+$ deciding that $\big((x,s),(y,t)\big)$ is an edge if either $x=y$ and $s\leq t$, or $s=t$ and $s$ is a jump time of $P_{x,y}$. Let $\{(w,s)\to_G(v,t)\}$ be the event that there is a path from $(w,s)$ to $(v,t)$ in $G$ (\textit{i.e.} a finite sequence of arrows at increasing times starting from $w$ and leading to $v$).
\begin{Def}
Let $\Pg_{W,\alpha}$ be the law of the process $(\zeta_t,~t\geq 0)$ on $X$ defined by $$\zeta_t(v)=1\Leftrightarrow\exists w\in W,~(w,0)\to_G(v,t),$$  and $\Eg_{W,\alpha}$ the expectation under that law. When we will use these notations we will continue to denote by $\zeta_t$ the value of the process at time $t$. If the law of some process is $\Pg_{W,\alpha}$ then we call it a $(W,\alpha)$-\emph{invasion process}. If $W=\{w\}$ with $v\in V$ we simply write $\Pg_{w,\alpha}$.
\end{Def}

It is clear from its definition that the invasion process possesses a monotonicity property w.r.t.\@ the parameter $\alpha$, namely if $\alpha$ and $\tilde{\alpha}$ are such that for $w\neq v$, $\alpha(w,v)\leq \tilde{\alpha}(w,v)$, then a simple coupling argument shows that for any measurable, positive function $f$ on $X$ which is increasing w.r.t.\@ the canonical partial order on $X$, we have 
\begin{equation}\Eg_{W,\alpha}[f(\zeta_t)]\leq \Eg_{W,\tilde\alpha}[f(\zeta_t)].\label{monotonie}\end{equation}
Let us define $$\gamma_\alpha(v,\chi)= (1-\chi(v))\sum_{w\neq v}\alpha(w,v)\chi(w),$$  $$g_\alpha(\chi)= \sum_{v\in V}\lambda_v\gamma_\alpha(v,\chi).$$ Note that for $\alpha=a$ we have $\gamma_a(v,\chi)\leq\lambda^{-1}\lambda_v\sum_{w\neq v}\lambda_v^{-1}\lambda_w a(w,v)$, so 
\begin{equation}\gamma_a(v,\chi)\leq \lambda^{-1}\lambda_v A.\label{gamma}\end{equation} We also define the function $$q(\chi)=\sum_{v\in V} \lambda_v \chi(v),~\chi\in X.$$ 
Note that if we take the weights $\lambda_v=1$ then $q(\chi)$ is the cardinality of $\chi$. Now we give a simple description of the invasion process in two special cases: first, when we start with a finite number of $0$'s, and then when we start with a $1$ at some site and $0$'s at all the other sites. In this paper we will only use these two initial conditions.
\begin{Prop}
\begin{itemize}
\item[\emph{(i)}] Suppose $W^c$ is finite. Under $\Pg_{W,a}$, $(\zeta_t,t\geq 0)$ is a Markov process on the finite set $X_W=\{\chi\in X : \restriction{\chi}{W}=1 \}$, starting from $\zeta_0=1_W$. Moreover, while in the configuration $\chi$, the site $v$ flips at rate $\gamma_a(v,\chi)$.
\item[\emph{(ii)}]Suppose now  that $W=\{w\}$ and (\ref{cond2}) is fulfilled. Then
\begin{equation}\Eg_{w,\overline a}[q(\zeta_t)]\leq\lambda_w e^{A t},\label{esperance}\end{equation}
$A$ being the constant which appears in (\ref{cond2}). In particular under $\Pg_{w,\overline a}$, $(\zeta_t, t\geq 0)$ is non-explosive in the sense that $\Pg_{w,\overline a}(\mathrm{Card}(\zeta_t)<+\infty)=1$, and $(\zeta_t,t\geq 0)$ is a Markov process taking its values in the countable set $X'$, starting from $\zeta_0=1_w$. Moreover, while in the configuration $\chi$, the site $v$ flips at a rate $\gamma_{\overline a}(v,\chi)$. 
\end{itemize}
\label{invasion}
\end{Prop}

\begin{proof}
In both cases the Markovian character of $(\zeta_t,t\geq 0)$ follows from the fact that Poisson processes have stationary independent increments.
\\To calculate the flip rates in the first case, just note that for an initial configuration $\chi$ with $\chi(v)=0$, the coordinate $\zeta_t(v)$ flips if one of the $P_{w,v}$ (with $\chi(w)=1$) jumps. The result follows from the fact that $\sum_{w:\chi(w)=1}P_{w,v}$ is a Poisson process with intensity $\gamma_\alpha(v,\chi)$.
\\For the second claim, first consider 
\begin{equation} a_n(x,y)=\begin{cases} a(x,y), & \text{ if }x,y\in V_n,   \\   0, & \text{ otherwise; }\end{cases}\label{notation}\end{equation}
and $\overline{a}_n(x,y)=a_n(y,x)$. Let $u_n(t)=\Eg_{w,\overline{a}_n}[q(\zeta_t)]$. On one hand, a simple calculation using (\ref{cond2}) shows that $g_{\overline{a}_n}(\chi)\leq A q(\chi)$. On the other hand under $\Pg_{w,\overline{a}_n}$, $(\zeta_t,~t\geq 0)$ takes its values in a finite state space so the following formula is easy to obtain: 
$$\Eg_{w,\overline{a}_n}[q(\zeta_{t+h})|\mathcal \zeta_t]=q(\zeta_t)+g_{\overline{a}_n}(\zeta_t)h+o(h).$$
Taking the expectation in this formula provides $$\lim_{h\to 0}\frac{u_n(t+h)-u_n(t)}{h}=\Eg_{w,\overline{a}_n}[g_{\overline{a}_n}(\zeta_t)]\leq A u_n(t).$$ Thus Gr\"onwall's lemma together with $u_n(0)=\lambda_w$ give the inequality $u_n(t)\leq\lambda_w e^{A t}$.
\\Now in order to complete the proof we describe a coupling of some processes $(\zeta^n,~n\geq 1)$ and $\zeta$ where $\zeta^n$ is a $(w,\overline{a}_n)$-invasion process, $\zeta$ is a $(w,\overline{a})$-invasion process, and $\zeta^n_t(v)$ is an increasing sequence that converges to $\zeta_t(v)$. To do this, we consider the above Poissonian construction and define $\zeta^n$ just like $\zeta$ except that $\zeta^n$ only uses $P_{x,y}$ with $x,y\in V_n$. More precisely we decide that $((x,s),(y,t))$ is an edge of the graph $G_n$ if either $x=y$ and $s\leq t$, or $s=t$, $x,y\in V_n$ and $s$ is a jump time of $P_{x,y}$. Then we define $\zeta^n_t$ by $$\zeta^n_t(v)=1\Leftrightarrow (w,0)\to_{G_n}(v,t).$$ Any given path of $G$ is also a path of $G_n$ as soon as $V_n$ is big enough to contain all the sites of the path. Consequently $\zeta_t(v)$ is the increasing limit as $n\to\infty$ of $\zeta^n_t(v)$. By the monotone convergence theorem, we then have $\Eg_{w,\overline{a}}[q(\zeta_t)]=\lim_{n\to\infty}u_n(t)$. Therefore the bound given for $u_n(t)$ still holds for $\Eg_{w,\overline{a}}[q(\zeta_t)]$.
\\ Since $\zeta_t$ does not explode, it is a jump process on a countable set. Therefore we can conclude as in the case where $W^c$ is finite to determine the flip rates.
\end{proof}
We point out that the above proof highlights the natural aspect of the assumption (\ref{cond2}). Indeed, if we take $\lambda_v=1$ for simplicity, (\ref{cond2}) exactly ensures that the invasion process with parameter $\overline a$ is non-explosive, by making it grow slower than a continuous-time branching process with intensity $A$. This point will turn out to be crucial in Corollary \ref{corollaire} below.\\ Now the following assertion establishes a duality between the invasion process with parameter $a$ and the one with parameter $\overline a$. 

\begin{Prop}
Let $v\in V$ and $W\subset V$. Then
\begin{equation}
\Pg_{W,a}(\zeta_t(v)=1)=\Pg_{v,\overline{a}}(\exists w\in W, \zeta_t(w)=1).
\label{dualite}
\end{equation}
\end{Prop}

\begin{proof}
Let $t\geq 0$, and for $s\leq t$ let $\tilde P_{x,y}(s)=P_{y,x}(t)-P_{y,x}(t-s)$. Clearly $\tilde P_{x,y}$ is a Poisson process on the time interval $[0,t]$ with intensity $a(y,x)$, and the processes $(\tilde P_{x,y},x\neq y)$ are mutually independant since the processes $(P_{x,y},x\neq y)$ are. We define another graph $\tilde G$ on $V\times[0,t]$ in the same way as $G$, but using $\tilde P_{x,y}$ instead of $P_{x,y}$. Equation (\ref{dualite}) then follows from the equivalence $$(w,0)\to_G(v,t)\Leftrightarrow (v,0)\to_{\tilde G}(w,t).$$
\end{proof}
We now turn to the main result of this section.
\begin{Th}
Let $n\geq 1$ and $\eta\in X$. There exists a coupling of processes $(\xi^{\eta,n},\eta\in X,n\geq 1)$ and $(\zeta^n, n\geq 1)$ such that
\begin{itemize}
\item[\emph{(i)}] $\xi^{\eta,n}$ is a $(\eta,V_n, c )$-finite spin system,
\item[\emph{(ii)}] $\zeta^n$ is a $(V_n^c,a)$-invasion process,
\item[\emph{(iii)}] for any $t\geq 0$ and $v\in V_n, $ we have $$\{\zeta^n_t(v)=0\}\subset\{\forall k\geq n,~\xi^{\eta,n}_t(v)=\xi^{\eta,k}_t(v)\}.$$
\end{itemize}
\label{couplage}
\end{Th}

\begin{proof}
As we announced, we are going to create a common graphical construction to build all the finite spin systems in the same probability space. Not surprisingly, our main ingredients will be a family $(N_v,~v\in V)$ of independent Poisson processes, $N_v$ having intensity $C+A\lambda^{-1}\lambda_v$, and a family $(U_{v,i},~v\in V,~i\geq 1)$ of independent random variables, $U_{v,i}$  being uniformly distributed over the interval $[0,C+A\lambda^{-1}\lambda_v]$. They are all defined on some appropriate probability space $(\Omega,\F,\PP)$, and the expectation under $\PP$ is denoted by $\EE$. For any $n\geq 1$ we consider $N^n=\sum_{v\in V_n}N_v$. Almost surely the jumps of $N^n$ are distinct and have no accumulation point, so there exists a strictly increasing sequence $(t_j,~j\geq 1)$ such that the jumps of $N^n$ are the $t_j,j\geq 1$. Let $v_j$ be the site of $V_n$ such that $N_{v_j}(t_j)-N_{v_j}(t_j^-)=1$, and $u_j=U_{v_j,N_{v_j}(t_j)}$. 
\\We then define $\xi^{\eta,n}_t$ on each interval of time $[t_j,t_{j+1})$ by induction:
\begin{itemize}
\item for $t\in[0,t_1)$, $\xi^{\eta,n}_t=\eta$;
\item for $t\in[t_j,t_{j+1})$, $\xi^{\eta,n}_t=\begin{cases} \left(\xi^{\eta,n}_{t_{j-1}}\right)^{v_j}, & \text{if } u_j<c_{v_j}(\xi^{\eta,n}_{t_{j-1}}),   \\   \xi^{\eta,n}_{t_{j-1}}, & \text{otherwise. }\end{cases}$
\end{itemize}
We also define $\zeta^n$ by induction, using the same Poisson processes:
\begin{itemize}
\item for $t\in[0,t_1)$, $\zeta^n_t=1_{V_n^c}$;
\item for $t\in[t_j,t_{j+1})$, let
\\$\zeta^n_t=
\begin{cases}
\left(\zeta^n_{t_{j-1}}\right)^{v_j}, 
& \text{if } \zeta^n_{t_{j-1}}(v_j)=0\text{, and }A^n_j\leq u_j\leq A^n_j+\gamma_a(v_j,\zeta^n_{t_{j-1}}),
\\   \zeta^n_{t_{j-1}}, & \text{otherwise;}
\end{cases}$
\end{itemize}
where $A^n_j=\inf_{k\geq n} c_{v_j}(\xi^{\eta,k}_{t_{j-1}})$. We recall that $\gamma_a(v_j,\zeta^n_{t_{j-1}})=\sum_{w\neq v_j}a(w,v_j)\zeta^n_{t_{j-1}}(w)$ when $\zeta^n_{t_{j-1}}(v_j)=0$.
\\It again follows from the properties of Poisson processes that each process $\xi^{\eta,n}$ and $\zeta^n$ has the Markov property. Moreover the flip rates are the ones required. Indeed, for each of these processes the flip rate of one coordinate is given by the length of the interval to which we ask $u_j$ to belong for this coordinate to flip, since in our construction this interval is always contained in $[0,C+A\lambda^{-1}\lambda_{v_j}]$ (see (\ref{gamma})). Just observe then that this length has been chosen properly to make the processes have the correct distribution.
\\Let us show (iii) by induction. With the convention $t_0=0$ we have to show that for any $j\geq 0$ and $v\in V$,\begin{equation}\{\zeta^n_{t_j}(v)=0\}\subset\{\forall k\geq n,~\xi^{\eta,n}_{t_j}(v)=\xi^{\eta,k}_{t_j}(v)\}.\label{HR}\end{equation}
The case $j=0$ is a consequence of the definition of the processes. Let us assume that (\ref{HR}) is fulfilled for some $j\geq 0$ and suppose that $\zeta^n_{t_j}(v_{j+1})=0$ (or else the conclusion is straightforward). Then (\ref{coincide}) gives $$\sup_{k\geq n} c_{v_{j+1}}(\xi^{\eta,k}_{t_j})\leq \inf_{k\geq n} c_{v_{j+1}}(\xi^{\eta,k}_{t_j})+\gamma_a(v_{j+1},\zeta^n_{t_j}).$$
Consequently, whenever the transition at time $t_{j+1}$ makes $\xi^{\eta,k}_t(v_{j+1})$ jump but not $\xi^{\eta,k'}_t(v_{j+1})$ for some $k,k'\geq n$, it must necessarily make also $\zeta^n_t(v_{j+1})$ jump from $0$ to $1$. Therefore (\ref{HR}) remains true for $j+1$.
\end{proof}
\begin{Cor}
In the previous coupling, for any $v\in V$ and $t\geq 0$, the sequence $\xi^{\eta,n}_t(v)$ is a.s.\@ constant beyond a certain (random) value of $n$, so we can define $$\xi^\eta_t(v)=\lim_{n\to\infty}\xi^{\eta,n}_t(v).$$
\label{corollaire}
\end{Cor}

\begin{proof}
Fix $v\in V$ and $t\geq 0$. The sequence of events $E_n=\{\forall k\geq n,~\xi^{\eta,k}_t(v)=\xi^{\eta,n}_t(v)\}$ is increasing so by Theorem \ref{couplage}(iii) it is enough to show that $\lim_{n\to\infty}\Pg_{V_n^c,a}$ $(\zeta_t(v)=1)=0$. The formula (\ref{dualite}) implies that 
\begin{align*}
\lim_{n\to\infty}\Pg_{V_n^c,a}(\zeta_t(v)=1) &=\lim_{n\to\infty}\Pg_{v,\overline a}(\exists w\in V_n^c : \zeta_t(w)=1) \\
&=\Pg_{v,\overline a}(\mathrm{Card}(\zeta_t)=+\infty)\\
&=0,
\end{align*}
where the last equality follows from (\ref{esperance}).
\end{proof}

\section{Generator and invariant measures for the limit process}

Up to now we only showed the existence of the limit process $(\xi^\eta_t,~t\geq 0)$ but we have no information about its law. In this section we show that it has the Markov property and that the expression of its generator is the one expected for functions in $\DX$. We adopt the notations and vocabulary from Chapter 1 of \cite{Liggett}.
\\For any function $f$ on $X$, let $\Delta_f(v)=\sup_{\eta\in X}|f(\eta)-f(\eta^v)|$, the maximal influence of the coordinate $\eta(v)$ on the value of $f(\eta)$. Rather than working with functions depending on a finite number of coordinates, we prefer to use the following space of \gui{good} functions, which turns out to be more natural for our purpose: $$\DX=\{f:X\to\RR,~ f\text{ is continuous and }\sum_{v\in V}\lambda_v\Delta_f(v)<+\infty\}.$$For $f\in\DX$ let $\tn f\tn=\sum_{v\in V}\lambda_v\Delta_f(v)$, and for $f$ bounded (which is the case if $f\in\DX$ since $X$ is a compact space) let $\Vert f\Vert_{\infty}=\sup_{\eta\in X}|f(\eta)|.$
The Markov pregenerator related to our dynamics is the operator $\Omega$ defined on $\DX$ by
$$\Omega f(\eta)=\sum_{v\in V}c_v(\eta)\big[f(\eta^v)-f(\eta)\big],~f\in\DX.$$ 
We will also consider
$$\Omega_n f(\eta)=\sum_{v\in V_n}c_v(\eta)\big[f(\eta^v)-f(\eta)\big],~f\in\DX,$$ its restriction on the finite volume $V_n$.
\\For any continuous function $f$ on $X$ and $\eta\in X$ we define $S_n(t)f(\eta)=\EE_{\eta,V_n}[f(\xi_t)]$ and $S(t)f(\eta)=\EE[f(\xi^\eta_t)]$. It easily follows from the bounded convergence theorem that \begin{equation}\lim_{n\to\infty}S_n(t)f(\eta)=S(t)f(\eta).\label{tcd}\end{equation}
Since we will use this theorem several times (on the probability space, on an interval of time, on the set $V$) we will refer to it as the BCT. We recall that there is a one-to-one correspondence between Markov semigroups and Markov generators. The result of this section is the following.
\begin{Th}
$S(t)$ defines a Markov semigroup, and for $f\in\DX$ we have 
\begin{equation}\lim_{t\to 0}\dfrac{S(t)f(\eta)-f(\eta)}{t}=\Omega f(\eta).\label{theor}\end{equation}
\end{Th}

\begin{proof}
Let us recall two identities for $S_n(t)$. For any $f\in\DX$ and $\eta\in X$ we have 
\begin{equation}S_n(t)f(\eta)=f(\eta)+\int_0^t\Omega_n S_n(s)f(\eta)\ud s,\label{fini1}\end{equation}
and for $t_1,~t_2\geq 0$ we have
\begin{equation}S_n(t_1)S_n(t_2)f(\eta)=S_n(t_1+t_2)f(\eta).\label{fini2}\end{equation}
Since both concern finite state space Markov processes they are elementary (see e.g. theorem 2.1.1 in \cite{Norris}). We begin by establishing some inequalities. First, for $f\in\DX$, 
\begin{equation}
|\Omega f(\eta)|\leq \lambda^{-1}C\tn f\tn,\text{ and } |\Omega_n f(\eta)|\leq \lambda^{-1}C\tn f\tn.
\label{omegaf}
\end{equation}
The first inequality simply follows from $|\Omega f(\eta)|\leq\sum_{v\in V}c_v(\eta)|f(\eta^v)-f(\eta)|\leq C\sum_{v\in V}\lambda^{-1}\lambda_v\Delta_f(v)$. A similar computation leads to the other one. Then, in order to control $\Delta_{S(t)f}(w)$, we consider two configurations $\eta_1$ and $\eta_2=(\eta_1)^w$ and we describe a new coupling based on an idea similar to the one in Section 1. We take $(N_v,~v\in V)$, $(U_{v,i},~v\in V,~i\geq 1)$ and the sequences $v_j$, $t_j$ and $u_j$ as in the proof of Theorem \ref{couplage}. Let $(\xi^{\eta_1,n} ,t\geq 0)$ and $(\xi^{\eta_2,n}, t\geq 0)$ be defined as follows:
\begin{itemize}
\item for $t\in[0,t_1)$, $\xi_t^{\eta_i,n}=\eta_i$;
\item for $t\in[t_j,t_{j+1})$, $\xi_t^{\eta_i,n}=\begin{cases}\left(\xi_{t_{j-1}}^{\eta_i,n}\right) ^{v_j},& \text{if } u_j<c_{v_j}(\xi_{t_{j-1}}^{\eta_i,n}),  \\ \xi_{t_{j-1}}^{\eta_i,n}, & \text{otherwise. }  \end{cases}$
\end{itemize}
We also define another process $(\Gamma^n_t,t\geq 0)$ using the same Poisson processes:
\begin{itemize}
\item for $t\in[0,t_1)$, $\Gamma^n_t=1_w$,
\item for $t\in[t_j,t_{j+1})$,
\\$\Gamma^n_t=
\begin{cases}
\left(\Gamma^n_{t_{j-1}}\right)^{v_j}, 
& \text{if } \Gamma^n_{t_{j-1}}(v_j)=0\text{, and } B^n_j\leq u_j\leq B^n_j+\gamma_{a_n}(v_j,\Gamma^n_{t_{j-1}}),   
\\   \Gamma^n_{t_{j-1}}, & \text{otherwise;}
\end{cases}$
\end{itemize}
where $B^n_j=\min_{i=1,2}c_{v_j}(\xi_{t_{j-1}}^{\eta_i,n})$. It is clear that $(\xi_t^{\eta_i,n},t\geq 0)$ is a $(\eta_i,V_n, c)$-finite spin system, and that $(\Gamma^n_t, t\geq 0)$ is a $(w,a_n)$-invasion process, keeping the notation (\ref{notation}). Moreover, for any site $v$, since the inclusion $\{\xi_t^{\eta_1,n}(v)\neq\xi_t^{\eta_2,n}(v) \ \}\subset\{ \Gamma^n_t(v)=1 \}$ is preserved by any possible transition and it is true at time $t=0$, it remains true for any $t\geq 0$. It then follows from (\ref{dualite}) and (\ref{monotonie}) that $$\PP(\xi_t^{\eta_1,n}(v)\neq\xi_t^{\eta_2,n}(v))\leq \PP(\Gamma^n_t(v)=1)=\Pg_{w,a_n}(\zeta_t(v)=1)\leq\Pg_{v,\overline{a}}(\zeta_t(w)=1).$$
\\The above coupling enables us to write
\begin{align*}
|S_n(t)f(\eta_1)-S_n(t)f(\eta_2)| &\leq\EE\big[|f(\xi^{\eta_1,n}_t)-f(\xi^{\eta_2,n}_t)|\big]\\
&\leq\EE\left[\sum_{v\in V}\Delta_f(v)\ind_{\{ \xi^{\eta_1,n}_t(v)\neq \xi^{\eta_2,n}_t (v)\}}\right]\\
&\leq\sum_{v\in V}\Delta_f(v)\PP( \xi^{\eta_1,n}_t(v)\neq \xi^{\eta_2,n}_t(v) ).
\end{align*}
Hence, considering this for any pair ($\eta_1,\eta_2$) of configurations that coincide everywhere off $w$, we get
 
\begin{equation}
\Delta_{S_n(t)f}(w)\leq\sum_{v\in V}\Delta_f(v)\Pg_{v,\overline{a}}(\zeta_t(w)=1).
\label{deltasntf}
\end{equation}
Then, by letting $n\to\infty$, we also get
\begin{equation}
\Delta_{S(t)f}(w)\leq\sum_{v\in V}\Delta_f(v)\Pg_{v,\overline{a}}(\zeta_t(w)=1).
\label{deltastf}
\end{equation}
Taking the sum over $w$ and interchanging the order of the sums provides thanks to (\ref{esperance}):
\begin{align*}
\tn S_n(t)f\tn &\leq\sum_{v\in V}\Delta_f(v)\sum_{w\in V}\lambda_w\Pg_{v,\overline{a}}(\zeta_t(w)=1)\\
&\leq\sum_{v\in V}\Delta_f(v)\Eg_{v,\overline{a}}[q(\zeta_t)]\\
&\leq\sum_{v\in V}\Delta_f(v)\lambda_v e^{At}.
\end{align*}
Therefore $S_n(t)f$ lies in $\DX$ if $f$ does, and we have the following inequality:
\begin{equation}
\tn S_n(t)f\tn\leq e^{At}\tn f\tn.
\label{ntsntf}
\end{equation}
Starting from (\ref{deltastf}) we get similarly:
\begin{equation}
\tn S(t)f\tn\leq e^{At}\tn f\tn.
\label{ntstf}
\end{equation}
For the definition of Markov semigroups we refer to Definition 1.4 of \cite{Liggett}. In order to show that $S(t)$ has that property, the only two non-obvious things to check are that for any $f\in\CX$, 
\begin{itemize}
\item[(I)] for $t_1,t_2\geq 0,~S(t_1)S(t_2)f=S(t_1+t_2)f$;
\item[(II)] the mapping $t\mapsto S(t)f$ is right-continuous on $\RR^+$.
\end{itemize}
It is enough to show that (I) and (II) are fulfilled for $f\in\DX$ since this is a dense subspace of $\CX$.
\\For (I) we first observe that the convergence in (\ref{tcd}) is actually uniform. Indeed, using the coupling introduced in Theorem \ref{couplage},
\begin{align*}
|S_n(t)f(\eta)-S(t)f(\eta)| &\leq\EE\big[|f(\xi^{\eta,n}_t)-f(\xi^\eta_t)|\big]\\
&\leq\EE\left[\sum_{v\in V}\Delta_f(v)\ind_{\{ \xi^{\eta,n}_t(v)\neq \xi^\eta_t (v)\}}\right]\\
&\leq\sum_{v\in V}\Delta_f(v)\PP( \xi^{\eta,n}_t(v)\neq \xi^\eta_t(v) )\\
&\leq\sum_{v\in V}\Delta_f(v)\Pg_{V_n^c,a}(\zeta_t(v)=1),
\end{align*}
which goes to $0$ by the BCT since $f\in\DX$ and $\lim_{n\to\infty}\Pg_{V_n^c,a}(\zeta_t(v)=1)=0$ (see the proof of Corollary \ref{corollaire}). Thus
\begin{equation}
\lim_{n\to\infty}\Vert S_n(t)f-S(t)f\Vert _{\infty}=0.
\label{sntf-stf}
\end{equation}
Now with (\ref{fini2}) we see that in order to obtain (I) for $f\in\DX$ it is enough to show that $\lim_{n\to\infty}\Vert S_n(t_1)S_n(t_2)f-S(t_1)S(t_2)f\Vert_\infty=0$. To do this we decompose:
\begin{align}
\Vert S_n(t_1)S_n(t_2)f-S(t_1)S(t_2)f\Vert_{\infty}&\leq \Vert (S_n(t_1)-S(t_1))S_n(t_2)f \Vert_{\infty} \nonumber\\
&\hspace{2cm}+ \Vert S(t_1)(S_n(t_2)-S(t_2))\Vert_{\infty}.
\label{decompose}
\end{align}
The first term in the right-hand side of (\ref{decompose}) is less than $$\sum_{v\in V}\Delta_{S_n(t_2)f}(v)\Pg_{V_n^c,a}(\zeta_t(v)=1).$$ As $n\to\infty$ this sum goes to $0$ by the BCT. Indeed by (\ref{deltasntf}) its summand is bounded by $\sum_{u\in V}\Delta_f(u)\Pg_{u,\overline{a}}(\zeta_{t_2}(v)=1)$, and the sum over $v\in V$ of this bound is finite because by (\ref{esperance}) it is less than $\lambda^{-1}\tn f\tn e^{A t}$. 
\\The second term also goes to $0$ because of (\ref{sntf-stf}) and the inequality $\Vert S(t_1)(S_n(t_2)-S(t_2))f\Vert_\infty \leq\Vert S_n(t_2)f-S(t_2)f\Vert _{\infty}$. Thus (I) is proved.
\\Now in order to prove (II) we begin by showing that for any $f\in\DX$,
\begin{equation}
S(t)f(\eta)=f(\eta)+\int_0^t\Omega S(s)f(\eta)\ud s.
\label{equagen}
\end{equation}
The integral is well defined because (\ref{omegaf}) and (\ref{ntstf}) give
\begin{equation}\label{omegasf}|\Omega S(s)f(\eta)|\leq \lambda^{-1} Ce^{As}\tn f\tn.\end{equation} 
Thanks to (\ref{fini1}) we just have to show that for each $s\geq0$, $\lim_{n\to\infty}\Omega_n S_n(s)f(\eta)=\Omega S(s)f(\eta)$. This together with the BCT --that we may use thanks to the bound $|\Omega_n S_n(s)f(\eta)|\leq \lambda^{-1}Ce^{As}\tn f\tn$, see (\ref{omegaf}) and (\ref{ntsntf})-- is enough to conclude. 
On one hand, 
\begin{align*}
|(\Omega_n-\Omega)S_n(s)f(\eta)|&\leq C\sum_{w\in V_n^c}\Delta_{S_n(s)f}(w)\\
&\leq C\sum_{w\in V_n^c}\sum_{v\in V}\Delta_f(v)\Pg_{v,\overline{a}}(\zeta_s(w)=1)\\
&= C\sum_{v\in V}\Delta_f(v) \sum_{w\in V_n^c}\Pg_{v,\overline{a}}(\zeta_s(w)=1),
\end{align*}
and for any $v\in V$ we have $\lim_{n\to\infty}\sum_{w\in V_n^c}\Pg_{v,\overline{a}}(\zeta_s(w)=1)=0$ since $$\sum_{w\in V}\Pg_{v,\overline{a}}(\zeta_s(w)=1)\leq\lambda^{-1}\Eg_{v,\overline{a}}[q(\zeta_s)]<+\infty,$$ so applying the BCT provides $\lim_{n\to\infty}|(\Omega_n-\Omega)S_n(s)f(\eta)|=0$. On the other hand, 
\begin{align*}
\Omega\big(S_n(s)-S(s)\big)f(\eta) &=\sum_{w\in V}c_w(\eta)\Big[\big(S_n(s)f(\eta^w)-S_n(s)f(\eta)\big)\\
&\hspace{2cm}-\big(S(s)f(\eta^w)-S(s)f(\eta)\big)\Big].
\end{align*}
In the right-hand side the summand goes to $0$ as $n\to\infty$ and by (\ref{deltasntf}) and (\ref{deltastf}) it is bounded by $C( \Delta_{S_n(s)f}(w)+\Delta_{S(s)f}(w))\leq 2C\sum_{v\in V}\Delta_f(v)\Pg_{v,\overline{a}}(\zeta_s(w)=1)$, which does not depend on $n$. Since the sum of this bound over $w$ is smaller than $2\lambda^{-1}Ce^{As} \tn f\tn<+\infty$, (\ref{equagen}) follows from the BCT. 
\\We are now able to conclude for (II) since (\ref{equagen}) together with (\ref{omegasf}) provide
\begin{equation}
\Vert S(t)f-f\Vert _\infty\leq\lambda^{-1}A^{-1}C(e^{At}-1)\tn f\tn,
\label{stf-f}
\end{equation}
which implies that $t\mapsto S(t)f(\eta)$ is continuous at $t=0$, and the continuity esaily extends to $\RR^+$ with (I).
\\In order to complete the proof of our theorem it remains to show that for $f\in\DX$ and $\eta\in X$, $$\lim_{t\to 0}\dfrac{S(t)f(\eta)-f(\eta)}{t}=\Omega f(\eta).$$
The mapping $t\mapsto \Omega S(t)f(\eta)$ is continuous at $t=0$. Indeed we can write 
\begin{align*}
\Omega\big(S(t+s)-S(t)\big)f(\eta) &=\sum_{w\in V}c_w(\eta)\Big[\big(S(t+s)f(\eta^w)-S(t+s)f(\eta)\big)\\
&\hspace{2cm}-\big(S(t)f(\eta^w)-S(t)f(\eta)\big)\Big],
\end{align*}
and conclude by an argument analogous to the one in the proof of (\ref{equagen}) using the continuity of $t\mapsto S(t)f(\eta)$. The continuity of  $t\mapsto \Omega S(t)f(\eta)$ together with (\ref{equagen}) directly imply $(\ref{theor})$.
\end{proof}

%-------------------------------------

Let $\mu$ be a probability measure on $X$. We will denote by $\mu S(t)$ the unique distribution on $X$ such that for any bounded measurable function $f$ on $X$, $\int f\ud[\mu S(t)]=\int S(t)f\ud\mu$. The measure $\mu$ is said \emph{invariant} if $\mu S(t)=\mu$ for any $t\geq 0$, what is equivalent to the condition $$\forall f\in\DX,~\int S(t)f\ud\mu=\int f\ud\mu.$$ We give a concrete criterion to check that $\mu$ is invariant. A function on $X$ is said to be \emph{local} if it depends only on a finite number of coordinates. 
\begin{Prop} 
Let $\mu$ be a probability measure on $X$. if $\int\Omega f\ud\mu=0$ for any local function $f$, then $\mu$ is invariant.
\end{Prop}
\begin{proof}
Theorem 5.2 in \cite{Liggett} states that the law of the process $(\xi^\eta_t,~t\geq 0)$ gives the unique solution to the martingale problem for $\Omega$ starting from  $\eta$. Thus the martingale problem is well-posed and consequently Proposition 6.10 of \cite{Liggett} gives the desired result
\end{proof}

\textbf{Acknowledgements.}\quad The author wishes to thank Enrique Andjel, Gr\'egory Maillard and Etienne Pardoux for their valuable comments and suggestions that considerably improved the style of the paper.

\end{document}